\documentclass{amsart}

\usepackage{amssymb,amsthm}
\usepackage{setspace} 
\usepackage{enumerate}
\usepackage{graphicx}
\usepackage{dsfont}
\usepackage{fullpage}
\usepackage{floatrow}
\usepackage{array}
\usepackage{hyperref}
\usepackage{bookmark}
\usepackage{xcolor}

\hypersetup
{
colorlinks,
linkcolor={red!50!black},
citecolor={blue},
urlcolor={blue!80!black}
}

\numberwithin{equation}{section}
\allowdisplaybreaks
\newtheorem{theorem}{Theorem}[section]
\newtheorem{lemma}[theorem]{Lemma}
\theoremstyle{definition}
\newtheorem{definition}{Definition}[section]

\newtheorem{corollary}{Corollary}[theorem]

\theoremstyle{proposition}
\newtheorem{proposition}{Proposition}[theorem]

\def\N{{\mathbb N}}

\def\R{{\mathbb R}}

\def\E{{\mathbb E}}

\def\J{{\mathfrak J}}

\def\O{{\mathcal O}}

\def\eps{\varepsilon}
\def\l{\ell}

\def\notagb{\notag \\}

\newcommand{\com}{,}
\newcommand{\ncomma}{, \ }
\newcommand{\ztp}{(0,2 \pi)}
\newcommand{\zp}{(0,\pi)}

\newcommand{\Envtwo}{\E [ N_{n} \ztp ]}

\newcommand{\asntoinfty}{\quad\mbox{as } n\to\infty}

\newcommand{\abs}[1]{\left\lvert #1 \right\rvert}

\newcommand{\ceil}[1]{\left\lceil #1 \right\rceil}
\newcommand\hw[1]{\noalign{\hrule height #1}}

\begin{document}

\title[Random trigonometric polynomials with pairwise equal blocks]{Real zeros of random trigonometric polynomials with pairwise equal blocks of coefficients}

\author{Ali Pirhadi}
\address{Department of Mathematics, Oklahoma State University, Stillwater, OK 74078, USA}
%
\email{pirhadi@math.okstate.edu}







\begin{abstract}
It is well known that the expected number of real zeros of a random cosine polynomial $ V_n(x) = \sum_ {j=0} ^{n} a_j \cos (j x) \ncomma x \in \ztp $, with the $ a_j $ being standard Gaussian i.i.d. random variables  is asymptotically $ 2n / \sqrt{3} $. On the other hand, some of the previous works on the random cosine polynomials with dependent coefficients show that such polynomials have at least $ 2n / \sqrt{3} $ expected real zeros lying in one period.
In this paper we investigate two classes of random cosine polynomials with pairwise equal blocks of coefficients. First, we prove that a random cosine polynomial with the blocks of coefficients being of a fixed length and satisfying $ A_{2j}=A_{2j+1} $ possesses the same expected real zeros as the classical case. Afterwards, we study a case containing only two equal blocks of coefficients, and show that in this case significantly more real zeros should be expected compared to those of the classical case.
\end{abstract}

\maketitle



\section{Introduction}
Let $ I \subset \R $ be an interval and $ f_0, f_1, \ldots , f_n : I \to \R $ be some functions. A (real) random function $ F_n: I \to \R $ is defined as the linear combination 
$
F_n (x) \equiv F_n (x,\omega):= \sum _{j=0} ^{n} a_j(\omega) f_{j}(x) ,
$
where the coefficients $ a_{j}(\omega) $ are real-valued random variables defined on the same probability space 
$ (\Omega, \mathcal{A}, \mathbb{P} ).$ 
A random (algebraic) polynomial of degree $ n $ is a function of the form
$
P_n(x) := \sum _{j=0} ^{n} a_{j} x^{j} ,
$
where the $ a_{j} $ are random variables. The study of zeros of polynomials with random coefficients, which is a venerable topic in both analysis and probability theory, has been a fascinating subject since the 1930s. The early work in this area was done by Bloch and P\'olya \cite{BP} where they concluded that the expected number of real zeros of a random polynomial with coefficients of the same probability selected from the set $ \{ -1, 0, 1 \} $ is $ \O \!  \left(\sqrt{n}\right)  $ for sufficiently large $ n $. Later,  Littlewood and Offord \cite{LO1}-\cite{LO2} studied the same estimate in which the coefficients are Gaussian and uniformly distributed  in  the set $ \{ -1, 1 \} $ or in the interval $ ( -1, 1 ) $. In the case of real Gaussian coefficients, one of the essential works is by Kac \cite{Kac2} where he introduces an explicit integral formula to determine the expected number of real zeros.
The early history of the subject as well as many additional references and further directions of work on the expected number of real zeros may be found in the books of Bharucha-Reid and Sambandham \cite{BS} and of Farahmand \cite{Far1}.

We refer to 
$
T_n(x) := \sum _{j=0} ^{n} a_{j} \cos (j x) + b_{j} \sin (j x) \ncomma x \in \ztp ,
$
as a random trigonometric polynomial of degree $ n $ with the coefficients $ a_{j} $ and $ b_{j} $ being random variables. In 1966, Dunnage \cite{Dun} showed that for a random cosine polynomial 
$
V_n(x) := \sum_{j=0}^{n} a_j \cos(j x) 
$
with the $ a_j $ being independent and identically distributed (i.i.d.) random variables with standard Gaussian distribution, we have
\begin{equation*} \label{1.5}
\Envtwo =\frac{2n}{\sqrt{3}} + \O \! \left( n^{ 11/13} (\log n )^{3/13}\right) \quad\mbox{as } n \to\infty \com
\end{equation*}
where $ \E $ is the mathematical expectation, and  $ N_{n} \ztp $ is the number of real zeros of $ V_n $ in $ \ztp $. 
Two years later, Das \cite{D} proved that the error term in Dunnage's result was $ \O \!  \left(\sqrt{n}\right)  $ for large $ n $, which was later significantly improved to $ \O (1) $ by Wilkins \cite{W}. There has also been a growing interest in both reducing and imposing different restrictions on the class of coefficients of a random trigonometric polynomial. For instance, the study of a certain class of nonzero mean coefficients appears in the work of Sambandham and Renganathan \cite{RS1} and that of a more general case of $ K $-level crossing with $ K \ne 0 $ and  $ \mu \ne 0 $ (the coefficients having nonzero mean) of Farahmand \cite{Far2}. 

When it comes to the case of random cosine polynomials with dependent coefficients, two cases are of great interest. Sambandham \cite{Sam}, and later  in a collaboration with Renganathan \cite{RS2}, investigated the cases of the constant correlation 
$ \E (a_i a_i ) = 1,$ $ \E (a_i a_j ) = \rho \in ( 0, 1 ) $
and the geometric correlation 
$ 
\E ( a_i a_j ) = \rho ^{  \abs{i - j} } ,
$ 
and showed in both cases the expected number of real roots asymptotically remain as $  2n / \sqrt{3}. $
More recently, the class of dependent coefficients has attracted more attention.
Angst, Dalmao and Poly \cite{ADP} proved that the expected number of real zeros of $ T_n  $ is asymptotically
$  
2n/\sqrt{3}
$ 
when $ a_0 = 0 $ and other coefficients are given by two independent stationary Gaussian processes with the same correlation function $ \rho $ and under mild assumptions of some spectral function associated with it. Another interesting direction related to the study of real zeros of random trigonometric polynomials with dependent coefficients is the work of Farahmand and Li \cite{FL}. The authors provide asymptotic estimates for the expected number of real zeros of polynomials $ T_n $ and $ V_n $ having \emph{palindromic} properties 
$ 
a_{j} = a_{ n - j } $ and $ b_{j} = b_{ n - j } .
$
According to their effort \cite[pp. 1880-1882]{FL} to show that 
$  
\Envtwo \sim n/ \sqrt{3} 
$ 
for the polynomial $ V_n $, we wish to capture the reader's attention that since the $ a_{j} $ are palindromic, for an odd $ n $, we have 
\begin{align*}
V_n  ( x ) 
& = \sum _{j=0} ^{n} a_j \cos (j x) 
= \sum _{j=0} ^{ m } a_j  \left(\cos (jx) + \cos (n-j)x\right)  
= 2 \cos ( nx/2 ) \cdot V^{*}_n (x) \com
\end{align*}
where $ m=(n-1)/2 $ and $ V^{*}_n (x) := \sum _{j=0} ^{ m } a_j \cos ( n/2 -j ) x $.
Let $ N_{n} \ztp $ and $ N^{*}_{n} \ztp $ denote the number of real zeros of $  V_n   $ and $  V^{*}_n   $ in $ \ztp $ respectively, so considering the $ n $ distinct deterministic roots of $ \cos ( nx/2 ) $,  in one period, the aforementioned result should be modified as
\begin{equation} \label{a1}
\Envtwo 
= \E  [n + N^{*}_{n} \ztp] 
= n + n / \sqrt{3} + \O \! \left( n^{3/4} \right) \asntoinfty \com
\end{equation}
which is entirely consistent with Proposition 2.1 of \cite{CFI}. 
The moral of the results mentioned above is that $ 2n / \sqrt{3} $ is the least number of real roots one can possibly expect when the independence of the coefficients is removed. 
The objective of this manuscript is to introduce two cases with dependent coefficients satisfying the above property.  
The main results of this paper are stated in Section \ref{Sec2} embarking with the general definition of a random function with pairwise equal blocks of coefficients and proceeding then with two theorems in each of which we deal with a random cosine polynomial where the independence setting on the coefficients is relaxed. Section \ref{Sec4} gives the proofs of the theorems appear in Section \ref{Sec2}.

\section{Random trigonometric polynomials with pairwise equal blocks of coefficients} \label{Sec2}
As a direct consequence of (\ref{a1}), it turns out that a random cosine polynomial with palindromic and normally distributed coefficients, on average, has almost $ 36.6\% $ more real zeros than that of the classical case with independent coefficients.
This motivates us to investigate how many real zeros, compared with the case of independent coefficients, we should expect if a certain restriction is imposed upon the coefficients, and more generally, with such a condition whether the number of these real zeros grows, drops or remains the same.
For instance, one can observe that if the coefficients of a random cosine polynomial of degree $ n $ are forced to have the following property $ a_{2j}=a_{2j+1} ,$ then the expected number of real zeros of such a polynomial is asymptotically $ 2n/\sqrt{3} ,$ as if the coefficients are independent.
Since any coefficient can be looked as a block of coefficient(s) of length $ 1 $, it is reasonable to generalize the previous questions and to study the number of real zeros when a certain restriction is applied to the \emph{blocks} of coefficients instead. 
In this section, we discuss two of such cases where the blocks of coefficients are pairwise equal in a given fashion.
\begin{definition}
A block of coefficients of length $ \l $ is defined as
$
A=  \left(a_{i}, a_{i+1}, \ldots , a_{i + \l -1}\right)  .
$
\end{definition}
Let $ n \in \N $ and 
$ 
A= \left(a_{0}, a_{1}, \ldots , a_{n}\right) ,
$ 
and assume that
$ 
n=2 \l m +r \ncomma r \in \{-1, \ldots , 2 \l -2 \} .
$ 
We may construct a sequence of $ 2m $ blocks of coefficients of length $ \l $ as 
$ 
\{ A_j \}_{j=0}^{2m-1} ,
$
and the set of the remaining coefficients
$ 
\tilde{A} = A \setminus \bigcup_{j} A_j  
$
containing $ r+1 $ elements.
\begin{definition}
Let $ n \in \N $ and
$ 
A= \left(a_{0}, a_{1}, \ldots , a_{n}\right) , 
$ and assume $ \{ A_j \}_{j=0}^{2m-1} $
and $ \tilde{A} $ as above. Suppose for each $ A_j  $ there exists a unique $ j' \neq j $ with $ A_{j} = A_{j'} ,$ and define
$
\J = \{ j=0 , \ldots , 2m-1 : A_{j} = A_{j'} \ \text{and} \ j < j' \}.
$
A random function $ F_n(x) = \sum _{j=0} ^{n} a_j f_j(x)  $ is called a random function with pairwise equal blocks of coefficients if $ \bigcup_{j \in \J } A_j \cup \tilde{A} $ is a family of i.i.d. random variables.
\end{definition}

Let $ n=2 \l m +r \ncomma r \in  \{-1, \ldots , 2 \l -2\} ,$
and $ V_n(x) = \sum _{j=0} ^{n} a_j \cos (j x) $ be a random cosine polynomial with pairwise equal blocks of coefficients. 
First, we consider a sequence of $ 2m $ blocks of coefficients of length $ \l $ in the following fashion: 
$  
A_0, A_{1}, \ldots, A_{2m-2}, A_{2m-1}, \tilde{A} ,
$  
where  
$ 	A_i=  \left(a_{\l i}, a_{\l i+1}, \ldots , a_{\l i + \l -1}\right) .$
We further assume that
$ A_{2j} = A_{2j+1} \ncomma j=0 , \ldots , m-1 ,$
and show that such a restriction does not affect the asymptotic expected number of real zeros of $ V_n $. In other words,
\subpdfbookmark{Theorem 2.1}{Thm2.1}
\begin{theorem} \label{Thm2.1}
Fix $ \l \in \mathbb{N} $, and let $ n= 2\l m +r \ncomma m \in \N ,$ 
and $ r \in \{-1, \ldots , 2 \l -2 \} .$
Assume $ V_n(x) = \sum _{j=0} ^{n} a_j \cos (j x) \ncomma x \in \ztp ,$
and $ \bigcup_{j=0}^{m-1} A_{2j} \cup \tilde{A} $ is a family of i.i.d. random variables with Gaussian distribution $\mathcal{N}(0, \sigma^2) .$ For $ j=0 , \ldots , m-1 $ and $ k=0 , \ldots , \l-1 ,$ we further assume $  a_{2 \l j+k} = a_{2 \l j + \l +k} ,$ i.e., $ A_{2j} = A_{2j+1}.$ Then
\begin{equation*}  
\Envtwo
=  \frac{2n}{ \sqrt{3}} + \O \! \left(n^{5/6}\right)  \asntoinfty.
\end{equation*}  
\end{theorem}

Similarly, if we define 
$ 
B= \left(b_{0}, b_{1}, \ldots , b_{n}\right) , 
$
$ 	B_i=  \left(b_{\l i}, b_{\l i+1}, \ldots , b_{\l i + \l -1}\right) $
and 
$ 
\tilde{B} = B \setminus \bigcup_{j=0}^{2m-1} B_j  ,
$
we have the following theorem which we will omit the proof since it follows the same procedure as the proof of Theorem \ref{Thm2.1}.
\currentpdfbookmark{Theorem 2.2}{Thm2.2}
\begin{theorem} \label{Thm2.2}
Fix $ \l \in \mathbb{N} $, and let $ n= 2\l m +r \ncomma m \in \N ,$ 
and $ r \in \{-1, \ldots , 2 \l -2 \} .$
Assume $ T_n(x) = \sum _{j=0} ^{n} a_{j} \cos (j x) + b_{j} \sin (j x) \ncomma x \in \ztp ,$
and $ \bigcup_{j=0}^{m-1} \left(A_{2j} \cup B_{2j}\right) \cup (\tilde{A} \cup \tilde{B})$ is a family of i.i.d. random variables with Gaussian distribution $\mathcal{N}(0, \sigma^2) .$ For $ j=0 , \ldots , m-1 $, we further assume $ A_{2j} = A_{2j+1}$ and $ B_{2j} = B_{2j+1}.$ Then
\begin{equation*}  
\Envtwo
=  \frac{2n}{ \sqrt{3}} + \O \! \left(n^{4/5}\right)  \asntoinfty.
\end{equation*}  
\end{theorem}

In the following theorem, we investigate the case of two equal blocks of coefficients of length $\ceil{n/2} $ as $ A_0, A_1, \tilde{A} ,$ where 
$ A_0=  \left(a_{0}, a_{1}, \ldots , a_{\ceil{n/2} -1}\right)$ and
$ A_1=  \left(a_{\ceil{n/2}}, a_{\ceil{n/2}+1}, \ldots , a_{2\ceil{n/2} -1}\right).$
Note that $ \tilde{A}=(a_n) $ if $ n $ is even, and it is empty otherwise.
\currentpdfbookmark{Theorem 2.3}{Thm2.3}
\begin{theorem} \label{Thm2.3}
Let $ V_n(x) = \sum _{j=0} ^{n} a_j \cos (j x) \ncomma n \in \N ,$ and $ x \in \ztp .$
Assume $ A_0 \cup \tilde{A} $ is a family of i.i.d. random variables with Gaussian distribution $\mathcal{N}(0, \sigma^2) .$ For $ j=0 , \ldots , m = \ceil{n/2} -1 ,$ we further assume $ a_{j} = a_{j + \ceil{n/2} }$, that is, $ A_{0} = A_{1}.$ Then
\begin{align*}  
\Envtwo
& = \frac{ n }{ 2  } + \frac{\sqrt{13} }{ 2 \sqrt{3}} \, n + \O  \! \left(n^{5/6}\right) \asntoinfty. 
\end{align*}  
\end{theorem}

In a similar way, the proof of the following theorem employs exactly the same machinery as that of Theorem \ref{Thm2.3} and is therefore omitted.
\currentpdfbookmark{Theorem 2.4}{Thm2.4}
\begin{theorem} \label{Thm2.4}
Let $ T_n(x) = \sum _{j=0} ^{n} a_{j} \cos (j x) + b_{j} \sin (j x)  \ncomma n \in \N ,$ and $ x \in \ztp .$
Assume $ A_0 \cup  B_0 \cup \tilde{A} \cup \tilde{B} $ is a family of i.i.d. random variables with Gaussian distribution $\mathcal{N}(0, \sigma^2) .$ If $ A_{0} = A_{1}$ and $ B_{0} = B_{1}$, then
\begin{align*} 
\Envtwo
& = \frac{ n }{ 2  } + \frac{\sqrt{13} }{ 2 \sqrt{3}} \, n + \O  \! \left(n^{4/5}\right) \asntoinfty. 
\end{align*}  
\end{theorem}

\section{Proofs} \label{Sec4}
First, it is worth estimating the expected real zeros of a random cosine polynomial in any negligible interval of length $ \O \! \left(n^{-a}\right) $. The advantage of the following lemma is its validity for any random cosine (trigonometric) polynomial with coefficients that are normally distributed and not necessarily independent.
\subpdfbookmark{Lemma 3.1}{Lem3.1}
\begin{lemma} \label{Lem3.1}
Let $ V_n(x) = \sum _{j=0} ^{n} a_j \cos (j x) $ with the $a_j$ being random variables with Gaussian distribution $\mathcal{N}(0, \sigma^2) $. If $ a \in (0, 1/2 ) $ is fixed, then 
$
\E \left[ N_{n} \! \left(0 , n^{-a}\right) \right] 
= \O  (n^{1-a})  
$
as $ n \to \infty. $	
\end{lemma}
\begin{proof}
For simplicity, we let $\sigma=1 .$ We note that 
$ V_n $ may be written as 
$ 
V_n (x) = e^ { -in x} P_{2n} ( e^{i x} ) ,  
$ 
where 
$ 
P_{2n} ( z ) := \sum_{ k = -n }^{  n } c_ { k } z ^{k+n}
$ 
with $ c_0 = a _0  $ and $ c_k = c_{-k} = a_{k} / 2 \ncomma k = 1, \ldots , n $. 
The fact that $ c_0  $ and the $ 2 c_k \ncomma k = \pm 1 , \ldots , \pm n $, are random variables with standard normal distribution suggests that 
\begin{align*}
& \E [  \abs{c_0} ] 
= \E [  \abs{2c_k} ] 
= \sqrt{ 2/\pi }
& & \text{and}  
& \E [ \log  \abs{c_0} ] 
= \E [ \log   \abs{2c_n} ] 
=   - (\gamma + \log{2})/2 \com 
\end{align*}
where
$ \gamma $ is the Euler-Mascheroni constant. So that
\begin{align*}
& M  := \sup \{ \E [ \abs{c_k} ] : k = 0 , \pm 1 , \ldots ,  \pm n \}  <  \infty
& & \text{and}  
& L := \inf  \{ \E [ \log \abs{c_k} ] : k = 0 \ \text{and} \ n \}  > - \infty . 
\end{align*}
For any complex random polynomial $ P_n $ of degree $ n $, we define the zero counting measure 
$ 
\tau _n := \frac{1}{n} \sum_{ k = 1 }^{n} \delta_{ z_k} ,
$ 
where $ \{z_k \}_{ k =0 }^{ n } $ are the zeros of $  P_n  $, and $ \delta_{ z_k} $ is the unit point mass at $ z_k $.
For any $ 0 < r <1 $, we also define the following annular sector
$
A_r ( \alpha , \beta ) = \{ z \in \mathbb{C} : r < | z | < 1/r , \ \alpha < \arg z < \beta \}.
$  
Set $ \alpha = 0, \ \beta = n^{-a}, \  r = 1/2 ,$ 
and let $ N_n (\, \cdot \,) $ and $ N_{2n}^{*} (\, \cdot \,)$ denote the number of zeros of $ V_n $ and $ P_{2n} $ respectively. It follows from Corollary 2.2. of \cite{PY} that 
\begin{equation*}
\E \left[ \,\abs{\tau_{ 2n }  \left( A_{ 1/2 } ( 0, n^{-a} ) \right) -  \frac{n^{-a} }{2 \pi}} \, \right] 
= \O  \! \left( \sqrt{  \frac{\log 2n}{ 2n}} \, \right) \quad\mbox{as }n\to\infty.
\end{equation*} 
So 
\begin{align*} 
\E [  N_{n}  ( 0 , n^{-a} )   ] 
& \leqslant \E \left[  N_{2n}^{*}  \left( A_{ 1/2 } ( 0 , n^{-a} ) \right)   \right] 
=  \O  \! \left(\sqrt{ n \log n}\right)  +  \O \! \left(n^{1-a}\right) = \O \! \left(n^{1-a}\right) \asntoinfty \com
\end{align*}
as required.
Likewise, 
$
\E [ N_{n} ( F ) ] = \O  \!  \left(n^ { 1-a }\right) 
$
holds for any interval $ F $ of length $ \O \! \left(n^{-a}\right) $ as $ n $ grows to infinity.
\end{proof}
An interested reader may need the following result while proving Theorems \ref{Thm2.2} and \ref{Thm2.4}.
\begin{proposition}
Let $ T_n(x) = \sum _{j=0} ^{n} a_{j} \cos (j x) + b_{j} \sin (j x) $ with the $ a_j $ and $ b_j $ being random variables with Gaussian distribution $\mathcal{N}(0, \sigma^2) $. If $ a \in (0, 1/2 ) $ is fixed, then 
$
\E \left[ N_{n} \! \left(0 , n^{-a}\right) \right] 
= \O  (n^{1-a})  
$
as $ n \to \infty. $
\end{proposition}
\begin{proof}
We note that 
$ 
T_n (x) = e^ { -in x} P_{2n} ( e^{i x} ) ,  
$ 
where 
$ 
P_{2n} ( z ) := \sum_{ k = -n }^{  n } c_ { k } z ^{k+n}
$ 
with $ c_0 = a _0  $ and $ c_k = \overline{c_{-k}} = (a_{k}-ib_{k}) / 2 \ncomma k = 1, \ldots , n .$
It is also clear that $ \E [  \abs{2c_k} ] = \E [  \abs{2c_{-k}} ]
\leqslant \E [ \abs{a_k} ] + \E [ \abs{b_k} ] $ and
$ \E [\log \abs{2c_n}] \geqslant \E [\log \, \max \left\{\abs{a_n},\abs{b_n}\right\} ] $
which imply that $ M <\infty $ and $ L > -\infty. $
The rest of the proof is similar to that of Lemma \ref{Lem3.1}.
\end{proof}
Before proving the first theorem, we need another lemma:
\begin{lemma}  \label{Lem4.1}
Fix $ a \in (0,1/2) $ and $ p \in \N ,$ 
and let 
$ 
x \in [ m^{-a}, \pi / p -m^{-a} ] \ncomma m \in \N.
$
For $ \lambda = 0 , 1 , 2 ,$ we define 
\begin{align*}
& P_{ \lambda } (p,m,x) 
:= \sum _{j=0} ^{m-1} j^{ \lambda }\cos(2p j)x 
& & \text{and}  
& Q_{ \lambda } (p,m,x) := \sum _{j=0} ^{m-1} j^ { \lambda } \sin(2p j)x \com
\end{align*}
then 
\begin{align*}
& P_{ \lambda } (p,m,x) = \O  \! \left(  m^ { \lambda+a} \right)
& & \text{and}  
& Q_{ \lambda } (p,m,x) = \O  \! \left(  m^ { \lambda+a} \right) \quad\mbox{as } m \to\infty .
\end{align*}
\end{lemma}

\begin{proof} 
The proof is similar to those of (2.1)-(2.6) of \cite[pp. 1877-1879]{FL} and by replacing $ \theta $ with $ p x $ and setting $ \eps = m^ {-a}. $
\end{proof}

\begin{corollary} \label{Cor4.1.1}
With the assumptions above, let 
\begin{align*}
& R_{ \lambda } (p,m,x) 
:= \sum _{j=0} ^{m-1} j^{ \lambda }\cos(2p j+1)x 
& & \text{and}  
& S_{ \lambda } (p,m,x) 
:= \sum _{j=0} ^{m-1} j^{ \lambda }\sin(2p j+1)x \com
\end{align*}
then by expanding $ \cos(2p j+1)x $ and $ \sin(2p j+1)jx $ along with the preceding lemma we have
\begin{align*} 
& R_{ \lambda } (p,m,x) 
= \O  \! \left(  m^ { \lambda+a} \right)
& & \text{and}  
& S_{ \lambda } (p,m,x) = \O  \! \left(  m^ { \lambda+a} \right) \quad\mbox{as } m \to\infty .
\end{align*}
\end{corollary}

Proving the desired theorems requires implementing the well-known Kac-Rice Formula established by Kac \cite{Kac2}, however we need a more generalized version of the formula, as in the work of Lubinsky, Pritsker and Xie \cite{LPX}.

\begin{proposition}[Kac-Rice Formula] \label{Prop4.1.1}
Let $ [a, b] \subset \R ,$ and consider real-valued functions $ f_{j}(x) \in C^1 [(a, b)] $,  $j = 0, \ldots , n $. Define the random function
$ 	F_n (x) := \sum _{j=0} ^{n} a_{j} f_{j}(x),  $
where the coefficients $a_j$ are i.i.d. random variables with Gaussian distribution $ \mathcal{N}(0,\sigma^{2}) $ and 
\begin{align*}
& A(x) := \sum _{j=0} ^{n} \left(f_{j} (x)\right)^2,
& B(x) := \sum _{j=0} ^{n} f_j (x) f'_{j} (x) \com 
& & \text{and}  &
& C(x) := \sum _{j=0} ^{n} \left(f'_{j} (x)\right)^2.
\end{align*}
If $ A(x)>0 $ on $ (a,b) ,$ and there is $ M \in \N $ such that $ F'_{n}(x) $
has at most $ M $ zeros in $ (a, b) $ for all choices of
coefficients, then the expected number of real zeros of $ F_n(x) $ in the interval $ (a, b) $ is given by
\begin{equation*}  
\E [ N_{ n }  ( a , b ) ] 
= \frac{1}{\pi} \displaystyle \int_{a}^{b}  \dfrac{ \sqrt{A(x)C(x)-B^2(x)} }{A(x)} \, dx.
\end{equation*}
\end{proposition}
\currentpdfbookmark{Proof of Theorem 2.1}{PThm2.1}
\begin{proof}[Proof of Theorem \ref{Thm2.1}] 
Fix $ a \in (0,1/5), $ and define $ \tilde{J} = \{ j : a_j \in \tilde{A}\} .$ 
For $ x \in \ztp $,
we see that
\begin{align*}
V_n(x)
& = \sum _{j=0} ^{n} a_j \cos (j x) 
= \sum_{j=0}^{m-1} \sum_{k=0}^{\l -1} a_{2\l j +k}  \left(\cos (2\l j +k) x + \cos (2\l j + \l +k ) x \right) 
+ \sum_{j \in \tilde{J}} a_j \cos (j x) \com 
\end{align*}
where 
$ 
\sum_{j \in \tilde{J}} a_j \cos (j x) = 0
$
if $ r=-1. $
Let $ x \in E = [ 0, \pi / \l ] \setminus F $ with $ F= [ 0,n^{-a}] \cup [ \pi / 2\l - n^{-a} , \pi / 2\l + n^{-a}] \cup [ \pi / \l - n^{-a} , \pi / \l ].$
We observe that
\begin{align*} 
A(x) 
& = \sum_{j=0}^{m-1} \sum_{k=0}^{\l -1} \left(\cos (2\l j +k) x + \cos ( 2\l j + \l +k ) x \right)^2  + \sum_{j \in \tilde{J}} \cos^2 (j x) \\
& =  4 \cos^2 (\l x/2) \sum_{j=0}^{m-1} \sum_{k=0}^{\l -1} \cos^2 (2\l j + \l/2 +k ) x 
+ \sum_{j \in \tilde{J}} \cos^2 (j x) .
\end{align*}
It is easy to examine that $ A(x)>0 $ on $ E $ since $ \cos (\l x/2) \neq 0 $ on $ E.$
Further,
\begin{align*} 
A(x) 
& = \sum_{j=0}^{m-1} \sum_{k=0}^{\l -1} \left(\cos (2\l j +k) x + \cos ( 2\l j + \l +k ) x \right)^2  + \sum_{j \in \tilde{J}} \cos^2 (j x) \\
& = \sum_{j=0}^{n} \cos^2 (j x)
+ 2 \sum_{j=0}^{m-1} \sum_{k=0}^{\l -1} \cos(2\l j +k) x \cos(2\l j+ \l +k ) x \\
& = \sum_{j=0}^{n} \cos^2 (j x)
+ \sum_{j=0}^{m-1} \sum_{k=0}^{\l -1} \left(\cos(\l x) + \cos(4\l j+\l +2k) x\right) \\
& = \frac{1}{2}\sum_{j=0}^{n} \left(1+ \cos (2j x)\right)
+ \l m \cos(\l x) + \sum_{j=0}^{m-1} \sum_{k=0}^{\l -1} \cos((4j+1)\l +2k) x \\
& = \frac{n+1}{2} + \frac{P_0(1,n+1,x)}{2} 
+ \frac{(n-r)\cos(\l x)}{2}  
+ \sum_{j=0}^{m-1} \sum_{k=0}^{\l -1} \cos (4j+1)\l  x \cos (2kx) \\
& - \sum_{j=0}^{m-1} \sum_{k=0}^{\l -1} \sin (4j+1)\l  x  \sin (2kx) \\
& = \frac{n+1}{2} + \frac{P_0(1,n+1,x)}{2} + \frac{(n-r)\cos(\l x)}{2} 
+ R_{ 0 } (2,m,\l x) \sum_{k=0}^{\l -1} \cos (2k) x  - S_{ 0 } (2,m,\l x) \sum_{k=0}^{\l -1} \sin (2k) x
\com
\end{align*}
and the boundedness of $ \sum_{k=0}^{\l -1} \cos (2k x) $ and $ \sum_{k=0}^{\l -1} \sin (2k x) $ along with Corollary \ref{Cor4.1.1} give that 
\begin{equation} \label{d1}
A(x) = \frac{n}{2} + \frac{n \cos(\l x)}{2} + \O  \! \left(n^ {a}\right)
= n \cos^2(\l x/2) + \O  \! \left(n^ {a}\right) \asntoinfty \ \text{and} \ x \in E.
\end{equation}
We also observe that
\begin{align*} 
C(x) 
& = \sum_{j=0}^{m-1} \sum_{k=0}^{\l -1} \left((2\l j +k) \sin (2\l j +k) x + (2\l j+\l +k) \sin (2\l j+\l +k) x \right)^2  + \sum_{j \in \tilde{J}} j^2 \sin^2 (j x) \notag \\
& = \sum_{j=0}^n j^2 \sin^2 (j x) 
+ 2 \sum_{j=0}^{m-1} \sum_{k=0}^{\l -1} (2\l j +k) (2\l j+\l +k) \sin(2\l j +k) x \sin(2\l j+\l +k) x \notag \\
& = \sum_{j=0}^n j^2 \sin^2 (j x) 
+ \sum_{j=0}^{m-1} \sum_{k=0}^{\l -1} (2\l j +k) (2\l j+\l +k) 
\left(\cos(\l x) - \cos(4\l j+\l +2k) x \right) \notag \\
& = \sum_{j=0}^n j^2 \sin^2 (j x) 
+ \cos(\l x) \sum_{j=0}^{m-1} \sum_{k=0}^{\l -1} (4\l^2 j^2 + 2\l^2 j + 4\l jk + \l k +k^2) \notag \\
& - \sum_{j=0}^{m-1} \sum_{k=0}^{\l -1} (4\l^2 j^2 + 2\l^2 j + 4\l jk + \l k +k^2) \cos (4j+1)\l  x \cos (2kx) \notag \\
& + \sum_{j=0}^{m-1} \sum_{k=0}^{\l -1} (4\l^2 j^2 + 2\l^2 j + 4\l jk + \l k +k^2) \sin (4j+1)\l  x \sin (2kx) .
\end{align*}
Moreover,
\begin{align*} 
\sum_{j=0}^{n} {j}^2 \sin ^2 (j x)
& = \frac{1}{2}\sum_{j=0}^{n} {j}^2 - \frac{1}{2} \sum_{j=0}^{n} {j}^2 \cos(2j x)
= \frac{ n ( n+1 ) ( 2n+1) }{12} - \frac{P_{ 2 } (1,n+1,x)}{2} 
= \frac{ n ^3 }{6} + \O  \! \left(n^{2+a}\right).
\end{align*}
Since $ \sum_{k=0}^{\l -1} k^{\lambda} =\O(1) \ncomma \lambda =1,2 ,$ and $ \sum_{j=0}^{m-1} j =\O  \! \left(m^2\right) $, a basic computation shows that
\begin{align*} 
& \sum_{j=0}^{m-1} \sum_{k=0}^{\l -1}  (4\l^2 j^2 + 2\l^2 j + 4\l jk + \l k +k^2)
=  4 \l ^ 3 \sum_{j=0}^{m-1} j^2 + \left(2\l ^3 + 4\l \sum_{k=0}^{\l -1} k \right) \sum_{j=0}^{m-1} j + \l m \sum_{k=0}^{\l -1} k 
+ m \sum_{k=0}^{\l -1} k^2 \notag \\
& = \frac{4 \l ^ 3 (m-1)m(2m-1) }{6} + \O  \! \left(m^{2}\right) 
= \frac{4\l^3 m^3 }{3} + \O  \! \left(m^{2}\right)
= \frac{4}{3}\left(\frac{n-r}{2}\right)^3 + \O  \! \left(m^{2}\right)
= \frac{n^3}{6} + \O  \! \left(n^2\right).
\end{align*}
We also note that
\begin{align*} 
& \sum_{j=0}^{m-1} \sum_{k=0}^{\l -1} (4\l^2 j^2 + 2\l^2 j + 4\l jk + \l k +k^2) \cos (4j+1)\l  x \cos (2kx) \notag \\
& = \l ^2 \left(4 R_{ 2 } (2,m,\l x) + 2 R_{ 1 } (2,m,\l x) \right) \sum_{k=0}^{\l -1} \cos (2kx)  + \l \left(4 R_{ 1 } (2,m,\l x) 
+ R_{ 0 } (2,m,\l x) \right) 
\sum_{k=0}^{\l -1} k \cos (2kx) \notag \\
& + R_{ 0 } (2,m,\l x) \sum_{k=0}^{\l -1} k^2 \cos (2kx) =  \O  \! \left(m^{2+a}\right) 
=  \O  \! \left(n^{2+a}\right),
\end{align*}
where the second-to-last equality comes from applying Corollary \ref{Cor4.1.1} and the fact that 
$ \sum_{k=0}^{\l -1} k^{\lambda} \cos (2kx) =\O(1)\ncomma \lambda =0, 1, 2.$
We likewise have 
\begin{align*} 
& \sum_{j=0}^{m-1} \sum_{k=0}^{\l -1} (4\l^2 j^2 + 2\l^2 j + 4\l jk + \l k +k^2) \sin (4j+1)\l  x \sin (2kx) =  \O  \! \left(n^{2+a}\right) .
\end{align*}
Thus, 
\begin{equation} \label{d6}
C(x) = \frac{n^3}{6} + \frac{n^3 \cos(\l x)}{6} + \O  \! \left(n^{2+a}\right)
= \frac{n^3 \cos^2(\l x/2)}{3} 
+ \O  \! \left(n^{2+a}\right) \asntoinfty \ \text{and} \ x \in E.
\end{equation}
We next observe that
\begin{align*} 
& B(x) 
= - \sum_{j \in \tilde{J}} j \sin(j x) \cos(j x)  \notag \\
&- \sum_{j=0}^{m-1} \sum_{k=0}^{\l -1} \left( \cos (2\l j +k) x + \cos ( 2\l j + \l +k ) x  \right) \left( (2\l j +k) \sin (2\l j +k) x + (2\l j+\l +k) \sin (2\l j+\l +k) x \right) \notagb	
& = - \sum_{j=0}^n j \sin(j x) \cos(j x)
- \sum_{j=0}^{m-1} \sum_{k=0}^{\l -1} (2\l j +k) \sin (2\l j +k) x \cos (2\l j + \l +k) x \notag \\
& - \sum_{j=0}^{m-1} \sum_{k=0}^{\l -1} (2\l j+\l +k) \sin (2\l j+\l +k) x \cos (2\l j +k) x	\notag \\
& = - \sum_{j=0}^n j \sin(j x) \cos(j x)
- \frac{1}{2} \sum_{j=0}^{m-1} \sum_{k=0}^{\l -1} (2\l j +k) 
\left( \sin(4\l j+\l +2k) x - \sin(\l x)\right) \notag \\
& - \frac{1}{2} \sum_{j=0}^{m-1} \sum_{k=0}^{\l -1} (2\l j+\l +k) 
\left( \sin(4\l j+\l +2k) x + \sin(\l x)\right)	\notag \\
& = - \sum_{j=0}^n j \sin(j x) \cos(j x)
- \frac{1}{2} \sum_{j=0}^{m-1} \sum_{k=0}^{\l -1} \l \sin(\l x) 
- \frac{1}{2} \sum_{j=0}^{m-1} \sum_{k=0}^{\l -1} (4\l j+\l +2k) 
\sin(4\l j+\l +2k) x \notag \\
& = - \sum_{j=0}^n j \sin(j x) \cos(j x)
- \frac{1}{2} \sum_{j=0}^{m-1} \sum_{k=0}^{\l -1} \l \sin(\l x) 
- \frac{1}{2} \sum_{j=0}^{m-1} \sum_{k=0}^{\l -1} (4\l j+\l +2k) 
\sin(4 j+1)\l x \cos(2kx) \notag \\
& - \frac{1}{2} \sum_{j=0}^{m-1} \sum_{k=0}^{\l -1} (4\l j+\l +2k) 
\cos(4 j+1)\l x \sin(2kx).	
\end{align*} 
It is evident that
\begin{align*} 
& \sum_{j=0}^{n} j \sin(j x) \cos(j x) 
= \sum_{j=0}^{n} (j /2) \sin(2j x) = \frac{Q_{1} (1,n+1,x)}{2} = \O  \! \left(n^{1+a}\right),
\end{align*} 
and that
\begin{align*} 
& \sum_{j=0}^{m-1} \sum_{k=0}^{\l -1} \l \sin(\l x) = \O  \! \left(m\right)
= \O  \! \left(n\right) .
\end{align*} 
We also note that
\begin{align*} 
& \sum_{j=0}^{m-1} \sum_{k=0}^{\l -1} (4\l j+\l +2k) 
\sin(4 j+1)\l x \cos(2kx) 
= \l \left( 4 S_{1} (2,m,\l x) + S_{0} (2,m,\l x)\right) 
\sum_{k=0}^{\l -1} \cos(2kx) \notag \\
& + 2 S_{0} (2,m,\l x) \sum_{k=0}^{\l -1} k \cos(2kx) 
= \O  \! \left(m^{1+a}\right) = \O \! \left(n^{1+a}\right) ,
\end{align*} 
where the second-to-last equality comes from applying Corollary \ref{Cor4.1.1} and the fact that 
$ \sum_{k=0}^{\l -1} k^{\lambda} \cos (2kx) =\O(1)\ncomma \lambda =0, 1.$
Similarly, 
\begin{align*} 
& \sum_{j=0}^{m-1} \sum_{k=0}^{\l -1} (4\l j+\l +2k) 
\cos(4 j+1)\l x \sin(2kx)  =  \O  \! \left(n^{1+a}\right) .
\end{align*}
Thus, 
\begin{equation} \label{d7}
B(x) = \O  \! \left(n^{1+a}\right) \asntoinfty \ \text{and} \ x \in E.
\end{equation}
Let $ y = \pi / \l - n^{-a} $, thus is clear that
\begin{align*}
\abs{\cos(\l x/2)} \geqslant \abs{\cos(\l y/2)}
= \sin\left( \l n^{-a}/2 \right)
\geqslant \l n^{-a}/\pi, \ x \in E,
\end{align*}
which implies that $ \sec(\l x/2) = \O  \! \left(n^{a}\right)$ on $ E .$ Now, (\ref{d1}), (\ref{d6}) and (\ref{d7}) give that
\begin{align*} 
& \sqrt{A(x)C(x) - B^2(x)}
= \sqrt{\frac{n^4 \cos^4(\l x/2)}{3} + \O  \! \left(n^{3+a}\right)} 
= \frac{n^2 \cos^2(\l x/2) \sqrt{1 + \sec^4(\l x/2) \O  \! \left(n^{-1+a}\right)}}{\sqrt{3}} \notag \\
& = \frac{n^2 \cos^2(\l x/2) \sqrt{1+ \O  \! \left(n^{-1+5a}\right)}}{\sqrt{3}} 
= \frac{n^2 \cos^2(\l x/2) \left(1+ \O  \! \left(n^{-1+5a}\right)\right)}{\sqrt{3}} \asntoinfty \ \text{and} \ x \in E,
\end{align*}
where the last equality holds since $ a \in (0,1/5). $
In a similar way, 
\begin{align*} 
A(x) = n \cos^2(\l x/2) \left(1+ \O  \! \left(n^{-1+3a}\right)\right) \asntoinfty \ \text{and} \ x \in E.
\end{align*}
So, the last two relations and Proposition \ref{Prop4.1.1} (\textit{Kac-Rice Formula}) give us
\begin{align*} 
\E [ N_{n} ( E ) ] 
& = \frac{1}{\pi} \displaystyle \int_{ E } \dfrac{n \left(1+ \O  \! \left(n^{-1+5a}\right)\right) }{\sqrt{3} \left(1+ \O  \! \left(n^{-1+3a}\right)\right)}  \, dx
=  \frac{n+ \O  \! \left(n^{5a}\right)}{\sqrt{3} \pi} \abs{E} 
= \frac{\left(n+ \O  \! \left(n^{5a}\right)\right)\left(\pi/\l + \O  \! \left(n^{-a}\right)\right)}
{\sqrt{3} \pi}  \notag \\
& =  \frac{n}{\sqrt{3} \l} +  \O  \! \left(n^{5a}\right) 
+  \O  \! \left(n^{1-a}\right) \asntoinfty.
\end{align*}
It also follows from Lemma \ref{Lem3.1} that
$  
\E \left[ N_{n} \! \left(F\right) \right]   
=  \O  \! \left(n^{1-a}\right) .
$
Therefore,
\begin{align*} 
\E [ N_{n} [0,\pi/\l] ] 
& =  \frac{n}{\sqrt{3} \l} +  \O  \! \left(n^{5a}\right) 
+  \O  \! \left(n^{1-a}\right) \asntoinfty ,
\end{align*}
and since $ a \in (0,1/5) ,$ the most best estimate takes place at $ a=1/6 $, thus 
\begin{align*} 
\E [ N_{n} [0,\pi/\l] ] 
& =  \frac{n}{\sqrt{3} \l} +  \O  \! \left(n^{5/6}\right) \asntoinfty .
\end{align*}
Similarly, we may conclude that 
$ 
\E [ N_{n} [ k\pi/\l, (k+1)\pi/\l  ] ] 
=  n/\sqrt{3} \l + \O  \! \left(n^{5/6}\right)
$
for $ k=1, \ldots , 2\l-1 $. Therefore,
\begin{align*} 
& \E [ N_{n} \ztp ] 
= 2 \l \, \E [ N_{n} [0,\pi/\l] ] 
=  \frac{2n}{ \sqrt{3}} + \O \! \left(n^{5/6}\right) \asntoinfty.
\end{align*} 

\end{proof}
\currentpdfbookmark{Proof of Theorem 2.3}{PThm2.3}
\begin{proof}[Proof of Theorem \ref{Thm2.3}] 
It is worthwhile to split the proof into two cases based on $ n $ being odd or even. Indeed, when $ n $ is odd, the proof is straightforward since we have plenty of deterministic zeros.

\emph{First case}: Say $ n $ is odd, so $ m= (n-1)/2 .$ For $ x \in \ztp $, we see that
\begin{align*}
V_n(x)
& = \sum _{j=0} ^{n} a_j \cos (j x)
= \sum_{j=0}^{m} a_{j}  \left( \cos (jx) + \cos (j+ (n+1)/2 ) x \right) 
= 2 \cos ( (n+1)x/4 ) \cdot V^{*}_n(x) \com
\end{align*}
where $ V^{*}_n(x):=\sum_{j=0}^{m}  a_j \cos (j+ (n+1)/4 )x .$
Let 
$ N_{n} \ztp $ and $ N^{*}_{n} \ztp $ be the number of real zeros of $  V_n  $ and $  V^{*}_n $ in $ \ztp $ respectively. Fix $ a \in (0,1/2) ,$ we first show that 
$ 
\E [N^{*}_{n} \zp ]
=  \left(\sqrt{13} / 4\sqrt{3}\right) n + \O  \! \left(n^{1-a}\right) 
$
as $ n \to \infty $.
Let $ x \in E = [0,\pi] \setminus F $ with $ F= [ 0,n^{-a} ] \cup [ \pi - n^{-a} , \pi ] $. We observe that 
\begin{align*} 
A(x) 
& = \sum_{j=0}^{m} \cos ^2 (j+ (n+1)/4 )x 
= \frac{m+1}{2} + \frac{1}{2} \sum_{j=0}^{m}  \cos ( 2j+ (n+1)/2 ) x.
\end{align*}
It follows from \cite[1.341(1,3) on p. 29]{GR} that, for all $ x \in \zp $,
\begin{align} \label{d9}
& \sum_{j=0}^{m-1}  \cos ( 2j + p ) x = \frac{\cos( m-1+p)x \sin(mx)}{\sin(x)}
& & \text{and}  
& \sum_{j=0}^{m-1}  \sin ( 2j + p ) x = \frac{\sin( m-1+p)x \sin(mx)}{\sin(x)} .
\end{align}
Hence, (\ref{d9}) helps us write
\begin{align*} 
A(x) & = \frac{ m+1}{2} \left(1 +  \frac{\cos ( nx ) \sin (m+1)x  }{(m+1) \sin(x)} \right) .
\end{align*}
Markov's inequality for algebraic polynomials \cite[Theorem 15.1.4. p. 567]{RS} gives 
$
\abs{\sin (m+1)x / \sin(x)} < m+1 , 
$
which implies that $ A(x)>0 $ on $ E. $ We also note that
\begin{align*} 
A(x) 
& = \frac{n}{4} 
+ \frac{\cos ( (n+1)x/2  )}{2} \sum_{j=0}^{m}  \cos ( 2jx )  
- \frac{\sin ( (n+1)x/2  )}{2} \sum_{j=0}^{m}  \sin ( 2jx ) \\
& = \frac{n}{4} 
+ \frac{\cos ((n+1)x/2) \, P_{0} (1,m+1,x)}{2}    
- \frac{\sin ((n+1)x/2) \, Q_{0} (1,m+1,x)}{2}  \com
\end{align*}
so by Lemma \ref{Lem4.1} we have
\begin{equation} \label{d10}
A(x) = \frac{n}{4} + \O  \! \left(n^{a}\right) 
= \frac{n \left( 1 + \O  \! \left(n^{-1+a}\right)\right)}{4} 
\asntoinfty \ \text{and} \ x \in E.
\end{equation}
We note that
\begin{align*} 
C(x) 
& = \sum_{j=0}^{m}  \left( j+ (n+1)/4 \right) ^2 \sin ^2 (j+ (n+1)/4 )x 
= \sum_{j=0}^{m} \left( j^2 /2 + (n+1)j/4  + (n+1)^2 /32 \right) \notag \\
& - \cos((n+1)x/2) \sum_{j=0}^{m}  \left( j^2 /2 + (n+1)j/4  + (n+1)^2 /32 \right)   \cos (2jx) \notag \\
& + \sin((n+1)x/2) \sum_{j=0}^{m}  \left( j^2 /2 + (n+1)j/4  + (n+1)^2 /32 \right)   \sin (2jx) .
\end{align*}
It is obvious that
\begin{align*} 
\sum_{j=0}^{m} \left( j^2 /2 + (n+1)j/4  + (n+1)^2 /32 \right)
& = \frac{m(m+1)(2m+1)}{12} + \frac{(n+1)m(m+1)}{8} +  \frac{(n+1)^2 m}{32} \notag \\
& = \frac{13n^ 3}{192} + \O  \! \left(n^{2}\right).
\end{align*}
Applying Lemma \ref{Lem4.1} gives us
\begin{align*} 
\sum_{j=0}^{m}  \left( j^2 /2 + (n+1)j/4  + (n+1)^2 /32 \right)   \cos (2jx) 
& = \frac{P_{2} (1,m+1,x) }{2} + \frac{(n+1) P_{1} (1,m+1,x)}{4} \notag \\
& +  \frac{(n+1)^2 P_{0} (1,m+1,x)}{32} 
= \O  \! \left(m^{2+a}\right) .
\end{align*}
In a similar way, we have
\begin{align*} 
& \sum_{j=0}^{m}  \left( j^2 /2 + (n+1)j/4  + (n+1)^2 /32 \right) \sin (2jx) = \O  \! \left(m^{2+a}\right) .
\end{align*}
Therefore, 
\begin{equation} \label{d15}
C(x) = \frac{13n^ 3}{192} + \O  \! \left(n^{2+a}\right) \asntoinfty \ \text{and} \ x \in E.
\end{equation}
In addition,
\begin{align*}
B(x) 
& = - \sum_{j=0}^{m} \left(j+ (n+1) /4 \right) \sin (j+ (n+1) /4 )x \, \cos ( j+ (n+1) /4 )x \\
& = - \cos ( (n+1)x/2 ) \sum_{j=0}^{m} \left(j /2+ (n+1) /8 \right) \sin (2jx)  
- \sin ( (n+1)x/2 ) \sum_{j=0}^{m} \left(j /2+ (n+1) /8 \right) \cos (2jx)  \\
& = - \cos ( (n+1)x/2 ) \left( \frac{Q_{1} (1,m+1,x)}{2} + \frac{(n+1) \, Q_{0} (1,m+1,x)}{8} \right) \\
& - \sin ( (n+1)x/2 ) \left( \frac{P_{1} (1,m+1,x)}{2} + \frac{(n+1) \, P_{0} (1,m+1,x)}{8} \right) , 
\end{align*}
hence,
\begin{equation} \label{d16}
B(x) = \O  \! \left(n^{1+a}\right) \asntoinfty \ \text{and} \ x \in E.
\end{equation}
Therefore, (\ref{d10}), (\ref{d15}) and (\ref{d16}) imply that
\begin{align*} 
\sqrt{ A(x) C(x) - B^2(x) }
& = \sqrt{ \frac{13  n^4}{768} + \O  \! \left(n^{3+a}\right) } 
= \frac{\sqrt{13} n^2 \left(1 + \O  \! \left(n^{-1+a}\right)\right) }{16 \sqrt{3}}\asntoinfty \ \text{and} \ x \in E,
\end{align*}
and it follows from the Kac-Rice Formula (Proposition \ref{Prop4.1.1}) that
\begin{align*} 
\E [ N^{*}_{n} ( E ) ]
& = \frac{1}{\pi} \displaystyle \int_{ E } 
\dfrac{\sqrt{13} n \left(1+ \O  \! \left(n^{-1+a}\right)\right) }{4 \sqrt{3} \left(1+ \O  \! \left(n^{-1+a}\right)\right)} \, dx 
= \dfrac{\sqrt{13} n \left(1+ \O  \! \left(n^{-1+a}\right)\right) }{4 \sqrt{3}\pi } \abs{E}
= \dfrac{\left(\sqrt{13} n + \O  \! \left(n^{a}\right) \right)\left( \pi +  \O  \! \left(n^{-a}\right) \right)}{4 \sqrt{3}\pi }   \\
& = \frac{ \sqrt{13} }{4 \sqrt{3}} \, n  + \O  \! \left(n^{1-a}\right) \asntoinfty .
\end{align*}
Since 
$  
\E [ N^{*}_{n} (F) ]
=  \O  \! \left(n^{1-a}\right) ,
$
we see that 
$ 
\E [ N^{*}_{n} \ztp ] = 2 \, \E [ N^{*}_{n} \zp ]
=  ( \sqrt{13} / 2\sqrt{3})n + \O  \! \left(n^{1-a}\right).
$
\noindent
Thus, considering the $ (n+1)/2 $ distinct (deterministic) roots of $ \cos ( (n+1)x/4 ) ,$  in one period, we have
\begin{equation} \label{d16-1}
\E [ N_{n} \ztp ]
= \E \left[ \frac{n+1}{2} + N^{*}_{n} \ztp \right] 
= \frac{ n }{ 2  } + \frac{\sqrt{13} }{ 2 \sqrt{3}} \, n 
+ \O  \! \left(n^{1-a}\right) \asntoinfty .
\end{equation} 

\emph{Second case}: Fix $ a \in (0,1/5) ,$ and let $ n $ be even, that is, $ m= n/2-1 .$ For $ x \in \ztp $, we see that
\begin{align*}
V_n(x)
& = \sum _{j=0} ^{n} a_j \cos (j x)
= \sum_{j=0}^{m} a_{j}  \left( \cos (jx) + \cos (j+ n/2 ) x \right) + a_n \cos(nx) \\
& = 2 \cos (nx/4 ) \sum_{j=0}^{m}  a_j \cos (j+ n/4 )x + a_n \cos(nx).
\end{align*}
Let $ x \in E = [0,\pi] \setminus F $ with $ F= [ 0,n^{-a} ] \cup [ \pi - n^{-a} , \pi ] $. With help of (\ref{d9}) we have
\begin{align*} 
A(x) 
& = 4\cos^2 (nx/4) \sum_{j=0}^{m} \cos ^2 (j+ n/4 )x + \cos^2(nx)
= 2\cos^2 (nx/4 ) \sum_{j=0}^{m} \left( 1+ \cos(2j+ n/2 )x \right) + \cos^2(nx) \notag \\
& = 2\cos^2 (nx/4) \left( \frac{n}{2} + \sum_{j=0}^{m} \cos(2j+ n/2 )x \right) + \cos^2(nx) \notag \\
& = 2\cos^2 (nx/4) \left(\frac{n}{2} + \frac{\cos(n-1)x \sin(nx/2)}{\sin(x)} \right) + \cos^2(nx) \notag \\
& = n\cos^2 (nx/4) \left(1 + \frac{2 \cos(n-1)x \sin(nx/2)}{n  \sin(x)} \right) + \cos^2(nx).
\end{align*}
Markov's inequality for algebraic polynomials \cite[Theorem 15.1.4. p. 567]{RS} gives 
$
\abs{\sin (nx/2) / \sin(x)} < n/2  ,
$
and since the zeros of $\cos(nx/4)$ and $ \cos(nx) $ do not coincide, $ A(x)>0 $ on $ E. $ Note that $ \sec(x) = \O  \! \left(n^{a}\right) $, $ x \in E, $ hence
\begin{align} \label{f2}
A(x) 
& = n\cos^2 (nx/4) + \cos^2(nx) + \frac{2 \cos(n-1)x \sin(nx/2) \cos^2 (nx/4)}{\sin(x)} \notag \\
& = n\cos^2 (nx/4) + \cos^2(nx) +  \sin(nx/2) \cos^2 (nx/4) \, \O  \! \left(n^{a}\right) \asntoinfty \ \text{and} \ x \in E.
\end{align}
We note that
\begin{align*} 
C(x) 
& = \sum_{j=0}^{m} \left( j\sin(jx) + (j+ n/2 ) \sin (j+ n/2 )x \right)^2  +  n^2 \sin^2 (nx) \notag \\
& = \sum_{j=0}^n j^2 \sin^2 (j x) 
+ 2 \sum_{j=0}^{m} j (j+ n/2) \sin(jx) \sin(j+ n/2) x \notag \\
& = \sum_{j=0}^n j^2 \sin^2 (j x) 
+ \sum_{j=0}^{m} (j^2+ nj/2) 
\left(\cos(nx/2) - \cos(2j+n/2) x \right) \notag \\
& = \sum_{j=0}^n j^2 \sin^2 (j x) 
+ \cos(nx/2) \sum_{j=0}^{m} (j^2+ nj/2) 
- \cos(nx/2) \sum_{j=0}^{m} (j^2+ nj/2) \cos(2jx) \notag \\
& + \sin(nx/2) \sum_{j=0}^{m} (j^2+ nj/2) \sin(2jx) .
\end{align*}
It is worthwhile to note that
\begin{align*} 
\sum_{j=0}^{n} {j}^2 \sin ^2 (j x)
& = \frac{1}{2}\sum_{j=0}^{n} {j}^2 - \frac{1}{2} \sum_{j=0}^{n} {j}^2 \cos(2j x)
= \frac{ n ( n+1 ) ( 2n+1) }{12} - \frac{P_{ 2 } (1,n+1,x)}{2} 
= \frac{ n ^3 }{6} + \O  \! \left(n^{2+a}\right) .
\end{align*}
It is clear that
\begin{align*} 
\sum_{j=0}^{m} (j^2+ nj/2)
& = \frac{m(m+1)(2m+1)}{6} + \frac{n m(m+1)}{4} 
= \frac{m^3}{3} + \frac{n m^2}{4} + \O  \! \left(m^{2}\right) 
= \frac{5 n^ 3}{48}+ \O  \! \left(n^{2}\right), 
\end{align*}
and that
\begin{align*} 
& \sum_{j=0}^{m} (j^2+ nj/2) \cos(2jx)
= P_{ 2 } (1,m+1,x) + n P_{ 1 } (1,m+1,x)/2  
=  \O  \! \left(m^{2+a}\right)
=  \O  \! \left(n^{2+a}\right).
\end{align*}
We likewise have
\begin{align*} 
& \sum_{j=0}^{m} (j^2+ nj/2) \sin(2jx) =  \O  \! \left(n^{2+a}\right) .
\end{align*}
Thus, 
\begin{align} \label{f8}
C(x) 
& = \frac{n^3}{6} + \frac{5n^3 \cos(n x/2)}{48} + \O  \! \left(n^{2+a}\right)
= \frac{n^3}{16} + \frac{5n^3 \cos^2(n x/4)}{24} 
+ \O  \! \left(n^{2+a}\right) \asntoinfty \ \text{and} \ x \in E.	
\end{align}
We also observe that
\begin{align*} 
B(x) 
& = - \sum_{j=0}^{m} \left( \cos (jx) + \cos (j+ n/2)x  \right) 	\left(j\sin(jx) + (j+ n/2 ) \sin (j+ n/2 )x\right) 
-  n \sin(nx) \cos(nx) \notag \\
& = - \sum_{j=0}^n j \sin(jx) \cos(jx)
- \frac{1}{2} \sum_{j=0}^{m} j \left( \sin(2j+ n/2)x - \sin(nx/2) \right) \notag \\
& - \frac{1}{2} \sum_{j=0}^{m} (j+ n/2) \left( \sin(2j+ n/2)x + \sin(nx/2) \right)	\notag \\
& = - \sum_{j=0}^n j \sin(j x) \cos(j x)
- \frac{1}{2} \sum_{j=0}^{m} (n/2) \sin(nx/2)
- \frac{1}{2} \sum_{j=0}^{m} (2j+ n/2) \sin(2j+ n/2)x \notag \\
& = - \sum_{j=0}^n j \sin(j x) \cos(j x)
- \frac{n^2 \sin(nx/2)}{8} 
- \frac{\cos(nx/2)}{2} \sum_{j=0}^{m} (2j+ n/2) \sin(2jx) \notag \\
& - \frac{\sin(nx/2)}{2} \sum_{j=0}^{m} (2j+ n/2) \cos(2jx).	
\end{align*} 
It is trivial to verify that
\begin{align*} 
& \sum_{j=0}^{n} j \sin(j x) \cos(j x) 
= \sum_{j=0}^{n} (j /2) \sin(2j x) = \frac{Q_{1} (1,n+1,x)}{2} = \O  \! \left(n^{1+a}\right),
\end{align*} 
and that
\begin{align*} 
& \sum_{j=0}^{m} (2j+ n/2) \sin(2jx) 
= 2 Q_{1} (1,m+1,x) + \frac{n\, Q_{0} (1,m+1,x)}{2}
= \O  \! \left(n^{1+a}\right) .
\end{align*} 
Comparably, 
\begin{align*} 
& \sum_{j=0}^{m} (2j+ n/2) \cos(2jx) 
= \O  \! \left(n^{1+a}\right).
\end{align*}
Thus, 
\begin{equation} \label{f13}
B(x) = - \frac{n^2 \sin(nx/2)}{8} + \O \!\left(n^{1+a}\right) \asntoinfty \ \text{and} \ x \in E.
\end{equation}
Let $ Z $ be the set of all zeros of $ \cos(nx/4) $ lying in $ E. $
Then, each element of $ Z $ may be written as $ x_k= (4k+2)\pi/n $ for some positive integer $ k $, and we have 
$ n/4 +  \O  \! \left(n^{1-a}\right)$
of such elements.
For $ x_k \in Z ,$ we define
$ \Omega'_k = \left[ x_k - n^{-1-a},x_k + n^{-1-a} \right] $, $ U= \cup_{x_k\in Z} \, \Omega'_k $ and
$ G = E \setminus U. $
It is also clear that $ \sec(nx/4) = \O  \! \left(n^{a}\right)$ on $ G $ since
\begin{align*}
\abs{\cos\left(nx/4 \right)} \geqslant \abs{\cos\left(n \left( x_k \pm n^{-1-a} \right)/4 \right)} = \sin\left(n^{-a}/4 \right) \geqslant n^{-a}/2\pi .
\end{align*}
It follows from (\ref{f2}) and the inequality above that
\begin{align} \label{f15}
A(x) 
& = n\cos^2 (nx/4) + \O  \! \left(1\right) 
+  \O  \! \left(n^{a}\right)
= n \cos^2((nx/4) + \O  \! \left(n^ {a}\right) \notagb
& = n \cos^2(nx/4) \left(1+ \O  \! \left(n^{-1+3a}\right)\right) \asntoinfty \ \text{and} \ x \in G.
\end{align}
A simple computation shows that
\begin{align} \label{f16}
& n\cos^2 (nx/4) \left(\frac{n^3}{16} + \frac{5n^3 \cos^2(n x/4)}{24}\right) - \left(\frac{n^2 \sin(nx/2)}{8}\right)^2
= \frac{n^4 \left(1+\cos(nx/2)\right) \left(8+5\cos(nx/2)\right)}{96}
\notag \\
& - \frac{n^4 \sin^2(nx/2)}{64}
= \frac{n^4 \left(13+26\cos(nx/2)+13\cos^2(nx/2)\right)}{192}
= \frac{13n^4\cos^4(nx/4)}{48}.
\end{align}
Therefore, (\ref{f8})-(\ref{f16}) give that
\begin{align*} 
& \sqrt{A(x)C(x) - B^2(x)}
= \sqrt{\frac{13n^4\cos^4(nx/4)}{48} + \O  \! \left(n^{3+a}\right)} 
= \frac{\sqrt{13}n^2\cos^2(nx/4) \sqrt{1 + \sec^4(nx/4) \O  \! \left(n^{-1+a}\right)} }{4\sqrt{3}} \notag \\
& = \frac{\sqrt{13}n^2\cos^2(nx/4) \sqrt{1+ \O  \! \left(n^{-1+5a}\right)} }{4\sqrt{3}}
= \frac{\sqrt{13}n^2\cos^2(nx/4) \left(1+ \O  \! \left(n^{-1+5a}\right)\right) }{4\sqrt{3}}\asntoinfty \ \text{and} \ x \in G,
\end{align*}
where the last equality holds since $ a \in (0,1/5). $
So, Proposition \ref{Prop4.1.1} gives us
\begin{align*} 
\E [ N_{n} ( G ) ] 
& = \frac{1}{\pi} \displaystyle \int_{G} \dfrac{\sqrt{13} n \left(1+ \O  \! \left(n^{-1+5a}\right)\right) }{4\sqrt{3} \left(1+ \O  \! \left(n^{-1+3a}\right)\right)}  \, dx
=  \frac{\sqrt{13} n+ \O  \! \left(n^{5a}\right)}{4\sqrt{3} \pi} \abs{G} 
= \frac{\left(\sqrt{13} n+ \O  \! \left(n^{5a}\right)\right)\left(\pi+ \O  \! \left(n^{-a}\right)\right)}{4\sqrt{3} \pi}  \notag \\
& =  \frac{\sqrt{13} n}{4\sqrt{3}} +  \O  \! \left(n^{5a}\right) 
+  \O  \! \left(n^{1-a}\right)\asntoinfty .
\end{align*}
It also follows from Lemma \ref{Lem3.1} that
$  
\E [ N_{n} (F) ]    
=  \O  \! \left(n^{1-a}\right) 
$
as $ n \to \infty. $ Therefore,
\begin{align*} 
\E [ N_{n} (G \cup F) ] 
& = \frac{\sqrt{13} n}{4\sqrt{3}} +  \O  \! \left(n^{5a}\right) 
+  \O  \! \left(n^{1-a}\right)\asntoinfty.
\end{align*}
We observe that the best estimate in above occurs at $ a=1/6 $ since $ a \in (0,1/5) .$ Set $a=1/6$ in all the computations so far. Thus, 
\begin{align} \label{f21}
\E [ N_{n} (G \cup F) ] 
& = \frac{\sqrt{13} n}{4\sqrt{3}} 
+ \O  \! \left(n^{5/6}\right) \asntoinfty ,
\end{align}
where $ G = E \setminus U $ and $ U= \bigcup_{x_k\in Z} \, \Omega'_k $ with $ \Omega'_k = \left[ x_k - n^{-7/6},x_k + n^{-7/6} \right]. $
To reach our objective, we need to show that
\begin{align} \label{f22}
\E \left[ N_{n} (U) \right] 
& = \frac{n}{4} + \O  \! \left(n^{5/6}\right) \asntoinfty .
\end{align}
Since we have $ n/4 + \O \!\left(n^{5/6}\right)$ of the $ \Omega'_k $ in $ (0,\pi) $, it makes sense to show that 
\begin{align*} 
\E \left[ N_{n} \left(\Omega'_k\right) \right] 
& = 1 + \O  \! \left(n^{-1/6}\right) \asntoinfty .
\end{align*}
Equivalently, we would like to prove that
\begin{align} \label{f24}
\E \left[ N_{n} \left(\Omega_k\right) \right] 
& = \frac{1}{2} + \O\!\left(n^{-1/6}\right) \asntoinfty ,
\end{align}
where $ \Omega_k = \left[ x_k,x_k + n^{-7/6} \right], $ and (\ref{f24}) is independent of any choice of the $x_k.$
Decomposing $ \Omega_k $ into four subintervals as 
$ \Omega_{k,1} = \big[ x_k+ n^{-5/4},x_k + n^{-7/6} \big], $ 
$ \Omega_{k,2} = \big[ x_k+ n^{-4/3},x_k + n^{-5/4} \big], $
$ \Omega_{k,3} = \big[ x_k+ n^{-5/3},x_k + n^{-4/3} \big] $
and $ \Omega_{k,4} = \big[ x_k,x_k + n^{-5/3} \big]$ provides a satisfactory path to reach (\ref{f24}). Given a small enough $ \eps >0 $, we also define 
$ \Omega_{k,1,\eps} = \big[ x_k+ n^{-5/4+\eps},x_k + n^{-7/6-\eps} \big]$ 
and 
$ \Omega_{k,2,\eps} = \big[ x_k+ n^{-4/3},x_k + n^{-5/4-\eps} \big].$  Table \ref{table1} will play a time-saving role in what follows.

\begin{table}[ht]
\renewcommand{\arraystretch}{1.5}
\begin{tabular}{!{\vrule width 1.5pt} l !{\vrule width 1pt} p{2cm}|p{2cm}|p{1.5cm}|p{1.5cm}!{\vrule width 1.5pt}}
\hw{1.5pt}
& \multicolumn{1}{c|}{on $\Omega_{k,1,\eps}$} & \multicolumn{1}{c|}{on $\Omega_{k,2,\eps}$} & \multicolumn{1}{c|}{on $\Omega_{k,3}$} & \multicolumn{1}{c!{\vrule width 1.5pt}}{on $\Omega_{k,4}$} \\
\hw{1pt}
$\cos(nx/4)$ & $\O\!\left(n^{-1/6-\eps}\right)$ & $\O\!\left(n^{-1/4-\eps}\right)$ & $\O\!\left(n^{-1/3}\right)$ & $\O\!\left(n^{-2/3}\right) $\\
\hline
$\sin(nx/2)$ & $\O\!\left(n^{-1/6-\eps}\right)$ & $\O\!\left(n^{-1/4-\eps}\right)$ & $\O\!\left(n^{-1/3}\right)$ & $\O\!\left(n^{-2/3}\right) $\\
\hline
$\sin(nx/2)\cos^2(nx/4)$ & $\O\!\left(n^{-1/2-3\eps}\right)$ & $\O\!\left(n^{-3/4-3\eps}\right)$ & $\O\!\left(n^{-1}\right)$ & $\O\!\left(n^{-2}\right)$\\
\hline
$n\cos^2(nx/4)$ & $\O\!\left(n^{2/3-2\eps}\right)$ & $\O\!\left(n^{1/2-2\eps}\right)$ & $\O\!\left(n^{1/3}\right)$ & $\O\!\left(n^{-1/3}\right)$\\
\hline
$n^3\cos^2(nx/4)$ & $\O\!\left(n^{8/3-2\eps}\right)$ & $\O\!\left(n^{5/2-2\eps}\right)$ & $\O\!\left(n^{7/3}\right)$ & $\O\!\left(n^{5/3}\right)$\\
\hline
$n^4\cos^4(nx/4)$ & $\O\!\left(n^{10/3-4\eps}\right)$ & $\O\!\left(n^{3-4\eps}\right)$ & $\O\!\left(n^{8/3}\right)$ & $\O\!\left(n^{4/3}\right)$\\
\hline
$n^2\sin(nx/2)$ & $\O\!\left(n^{11/6-\eps}\right)$ & $\O\!\left(n^{7/4-\eps}\right)$ & $\O\!\left(n^{5/3}\right)$ & $\O\!\left(n^{4/3}\right)$\\
\hline
$\sec(nx/4)$ & $\O\!\left(n^{1/4-\eps}\right)$ & $\O\!\left(n^{1/3}\right)$ & $\O\!\left(n^{2/3}\right)$ & \\
\hw{1.5pt}
\end{tabular} 
\caption{Approximations on the given subintervals}
\label{table1}
\end{table}	
First, we show that
\begin{align} \label{f25}
\E \left[ N_{n} \left(\Omega_{k,1}\right) \right] 
& = \O\!\left(n^{-1/6}\right) \asntoinfty .
\end{align}
One may verify that (\ref{f2}), (\ref{f8}) and (\ref{f13}) along with Table \ref{table1} imply that
\begin{align*} 
A(x) 
& = n\cos^2 (nx/4) + \O\!\left(1\right) + \O\!\left(n^{-1/3-3\eps}\right)
= n\cos^2 (nx/4) + \O\!\left(1\right) \notagb
& = n \cos^2(nx/4) \left(1+ \O  \! \left(n^{-1/2-2\eps}\right)\right) 
\asntoinfty \ \text{and} \ x \in \Omega_{k,1,\eps}.
\end{align*}
\begin{align*} 
C(x) = \frac{n^3}{16} + \frac{5n^3 \cos^2(n x/4)}{24} 
+ \O  \! \left(n^{13/6}\right) \asntoinfty \ \text{and} \ x \in \Omega_{k,1,\eps}
\end{align*}
and
\begin{align*} 
B(x) = - \frac{n^2 \sin(nx/2)}{8} + \O\!\left(n^{7/6}\right) \asntoinfty \ \text{and} \ x \in \Omega_{k,1,\eps}.
\end{align*}
Then, (\ref{f16}) and Table \ref{table1} give us
\begin{align*} 
\sqrt{A(x)C(x) - B^2(x)}
& = \sqrt{\frac{13n^4\cos^4(nx/4)}{48} + \O  \! \left(n^{3}\right)} 
= \frac{\sqrt{13}n^2\cos^2(nx/4) \sqrt{1 + \sec^4(nx/4) \O  \! \left(n^{-1}\right)} }{4\sqrt{3}} \notag \\
& = \frac{\sqrt{13}n^2\cos^2(nx/4) \left(1+ \O  \! \left(n^{-4\eps}\right)\right) }{4\sqrt{3}}\asntoinfty \ \text{and} \ x \in \Omega_{k,1,\eps},
\end{align*}
So, the Kac-Rice Formula implies that
\begin{align*} 
\E [ N_{n} \left(\Omega_{k,1,\eps}\right) ] 
& = \frac{1}{\pi} \displaystyle \int_{\Omega_{k,1,\eps}} 
\dfrac{\sqrt{13} n \left(1+ \O  \! \left(n^{-4\eps}\right)\right) dx}{4\sqrt{3} \left(1+ \O  \! \left(n^{-1/2-2\eps}\right)\right)}  
=  \dfrac{\sqrt{13} n + \O  \! \left(n^{1-4\eps}\right) }{4\sqrt{3} \pi} \abs{\Omega_{k,1,\eps}} \notag \\
& = \O\!\left(n\right)   \O\!\left(n^{-7/6-\eps}\right)
=  \O  \! \left(n^{-1/6-\eps}\right)\asntoinfty .
\end{align*}
(\ref{f25}) then follows by letting $ \eps \to 0^+. $

Second step requires showing that
\begin{align} \label{f31}
\E \left[ N_{n} \left(\Omega_{k,2}\right) \right] 
& = \O\!\left(n^{-1/6}\right) \asntoinfty .
\end{align}
It follows from (\ref{f2}), (\ref{f8}), (\ref{f13}) and the above table that
\begin{align*} 
A(x) 
& = n\cos^2 (nx/4) + \cos^2(nx) + \O\!\left(n^{-7/12-3\eps}\right)
\asntoinfty \ \text{and} \ x \in \Omega_{k,2,\eps},
\end{align*}
\begin{align*} 
C(x) 
& = \frac{n^3}{16} + \frac{5n^3 \cos^2(n x/4)}{24} 
+ \O  \! \left(n^{13/6}\right) 
= \frac{n^3}{16} + \O\!\left(n^{5/2-2\eps}\right)
+  \O  \! \left(n^{13/6}\right) \notag \\
& = \frac{n^3}{16} + \O\!\left(n^{5/2-2\eps}\right)
\asntoinfty \ \text{and} \ x \in \Omega_{k,2,\eps}
\end{align*}
and
\begin{align*} 
B(x) = - \frac{n^2 \sin(nx/2)}{8} + \O\!\left(n^{7/6}\right) \asntoinfty \ \text{and} \ x \in \Omega_{k,2,\eps}.
\end{align*}
The same argument as (\ref{f16}) can easily be adapted to see
\begin{align*} 
& \frac{n^4 \cos^2 (nx/4)}{16} - \left(\frac{n^2 \sin(nx/2)}{8}\right)^2
= \frac{n^4 \cos^4 (nx/4)}{16}.
\end{align*}
Therefore, the relations above and Table \ref{table1} imply that
\begin{align*} 
\sqrt{A(x)C(x) - B^2(x)}
& = \sqrt{\frac{n^3\cos^2(nx)}{16} + \frac{n^4 \cos^4 (nx/4)}{16}
+ \O  \! \left(n^{3-4\eps}\right)} 
= \sqrt{\frac{n^3\cos^2(nx)}{16} + \O  \! \left(n^{3-4\eps}\right)} \notag \\
& = \frac{n^{3/2}\cos(nx) \left(1+ \O  \! \left(n^{-4\eps}\right)\right)}{4} \asntoinfty \ \text{and} \ x \in \Omega_{k,2,\eps}.
\end{align*}
Thus, with help of Kac-Rice's Formula we have
\begin{align*} 
\E [ N_{n} \left(\Omega_{k,2,\eps}\right) ] 
&= \frac{1}{\pi} \displaystyle \int_{\Omega_{k,2,\eps}} 
\dfrac{n^{3/2}\cos(nx) \left(1+ \O  \! \left(n^{-4\eps}\right)\right) dx }
{4 \left(n \cos^2(nx/4) + \cos^2(nx)\right) \left(1+ \O  \! \left(n^{-7/12-3\eps}\right)\right)} \notag \\
& =  \dfrac{n^{3/2}\left(1+ \O  \! \left(n^{-4\eps}\right)\right) }{4\pi} \displaystyle \int_{\Omega_{k,2,\eps}} 
\dfrac{\cos(nx) \, dx}
{n \cos^2(nx/4) + \cos^2(nx)} .
\end{align*}
For the moment, letting $\eps \to 0^+$ we obtain
\begin{align*} 
\E [ N_{n} \left(\Omega_{k,2}\right) ] 
& =  \O \!\left(n^{3/2}\right)
\int_{\Omega_{k,2}} 
\dfrac{\cos(nx) \, dx}
{n \cos^2(nx/4) + \cos^2(nx)} .
\end{align*}
Change of variables $ u=n^{3/2}\left( x-x_k \right) $ assists us to write
\begin{align*} 
0 \leqslant \int_{\Omega_{k,2}} \dfrac{\cos(nx) \, dx}{n \cos^2(nx/4) + \cos^2(nx)} 
& \leqslant \int_{x_k+n^{-4/3}}^{x_k+n^{-5/4}} \dfrac{dx}{n \cos^2(nx/4)} 
= n^{-3/2} \int_{n^{1/6}}^{n^{1/4}} \dfrac{dx}{n \sin^2 \left(x/4\sqrt{n}\right)} \notag \\
& \leqslant 4 \pi^2 n^{-3/2} \int_{n^{1/6}}^{n^{1/4}} \dfrac{dx}{x^2}
= 4 \pi^2 n^{-3/2} \left( n^{-1/6} -n^{-1/4} \right).
\end{align*}
Hence, the last two relations guarantee that (\ref{f31}) holds.

In the third place, we would like to prove that
\begin{align} \label{f39}
\E \left[ N_{n} \left(\Omega_{k,3}\right) \right] 
& = \frac{1}{2}+ \O\!\left(n^{-1/6}\right) \asntoinfty .
\end{align}
We follow the same procedure as before and observe that
\begin{align*} 
A(x) 
& = n\cos^2 (nx/4) + \cos^2(nx) + \O\!\left(n^{-5/6}\right)
\asntoinfty \ \text{and} \ x \in \Omega_{k,3},
\end{align*}
\begin{align*} 
C(x) 
& = \frac{n^3}{16} + \frac{5n^3 \cos^2(n x/4)}{24} 
+ \O  \! \left(n^{13/6}\right) 
= \frac{n^3}{16} + \O\!\left(n^{7/3}\right)
+  \O  \! \left(n^{13/6}\right) \notag \\
& = \frac{n^3}{16} + \O\!\left(n^{7/3}\right)
\asntoinfty \ \text{and} \ x \in \Omega_{k,3},
\end{align*}
and that
\begin{align*} 
B(x) = - \frac{n^2 \sin(nx/2)}{8} + \O\!\left(n^{7/6}\right) \asntoinfty \ \text{and} \ x \in \Omega_{k,3}.
\end{align*}
Hence, 
\begin{align*} 
\sqrt{A(x)C(x) - B^2(x)}
& = \sqrt{\frac{n^3\cos^2(nx)}{16} + \frac{n^4 \cos^4 (nx/4)}{16}
+ \O  \! \left(n^{8/3}\right)} 
= \sqrt{\frac{n^3\cos^2(nx)}{16} + \O  \! \left(n^{8/3}\right)} \notag \\
& = \frac{n^{3/2}\cos(nx) \left(1+ \O  \! \left(n^{-1/3}\right)\right)}{4} \asntoinfty \ \text{and} \ x \in \Omega_{k,3}.
\end{align*}
So, the Kac-Rice Formula together with change of variables $ u=n^{3/2}( x-x_k ) $ give us
\begin{align} \label{f46}
\E [ N_{n} \left(\Omega_{k,3}\right) ] 
& = \frac{1}{\pi} \displaystyle \int_{\Omega_{k,3}} 
\dfrac{n^{3/2}\cos(nx) \left(1+ \O  \! \left(n^{-1/3}\right)\right) dx }
{4 \left(n \cos^2(nx/4) + \cos^2(nx)\right) \left(1+ \O  \! \left(n^{-5/6}\right)\right)} \notag \\
& =  \dfrac{1+ \O  \! \left(n^{-1/3}\right)}{4\pi} \displaystyle \int_{\Omega_{k,3}} 
\dfrac{n^{3/2} \cos(nx) \, dx}
{n \cos^2(nx/4) + \cos^2(nx)} \notag \\
& = \dfrac{1+ \O  \! \left(n^{-1/3}\right)}{4\pi} \int_{n^{-1/6}}^{n^{1/6}} \dfrac{\cos\left(x/\sqrt{n}\right)\, dx}{n \sin^2\left(x/4\sqrt{n}\right) + \cos^2\left(x/\sqrt{n}\right)} .
\end{align}
We define 
\begin{align*}
f_{n}(x):=  \dfrac{\cos\left(x/\sqrt{n}\right)}
{n \sin^2\left(x/4\sqrt{n}\right) + \cos^2\left(x/\sqrt{n}\right)}  \cdot \mathds{1}_{\left[ n^{-1/6},n^{1/6} \right]} (x) \ncomma x \in [0, \infty).
\end{align*}  
It is obvious that $ \sin\left(x/4\sqrt{n}\right) \geqslant x/2\pi\sqrt{n} $ and $ \cos\left(x/\sqrt{n}\right) \geqslant 1/2 $ on $ \left[ n^{-1/6},n^{1/6} \right].$ Therefore,
\begin{align*}
0 \leqslant f_{n}(x) \leqslant  \dfrac{4\pi^2}{x^2+\pi^2}
& =: g(x) \ \text{and } g\in L^1 \! \left(\left[0,\infty\right)\right).
\end{align*}  
Therefore, Lebesgue's Dominated Convergence Theorem says that
\begin{align} \label{f47}
& \lim_{n \to \infty} \int_{n^{-1/6}}^{n^{1/6}} \dfrac{\cos\left(x/\sqrt{n}\right)\, dx}{n \sin^2\left(x/4\sqrt{n}\right) + \cos^2\left(x/\sqrt{n}\right)} 
=16 \int_{0}^{\infty} \dfrac{dx}{x^2+16}=2\pi.
\end{align}
We then use change of variables $ u=x/\sqrt{n}$ and observe that
\begin{align*} 
& 0 \leqslant \int_{n^{-1/6}}^{n^{1/6}} \dfrac{\cos\left(x/\sqrt{n}\right)\, dx}{n \sin^2\left(x/4\sqrt{n}\right) + \cos^2\left(x/\sqrt{n}\right)} 
\leqslant \int_{n^{-1/6}}^{n^{1/6}} \dfrac{dx}{n \sin^2\left(x/4\sqrt{n}\right) + \cos^2\left(x/\sqrt{n}\right)} \notag \\
& = \sqrt{n} \int_{n^{-2/3}}^{n^{-1/3}} \dfrac{dx}{n \sin^2\left(x/4\right) + \cos^2\left(x\right)}	.
\end{align*}
The convexity of the sine function on $ \left[0, n^{-1/3}\right] $ gives us $ \sin(x) \geqslant x n ^{1/3} \sin\left(n ^{-1/3}\right),$  which consequently implies that
\begin{align*} 
& \int_{n^{-2/3}}^{n^{-1/3}} \dfrac{dx}{n \sin^2\left(x/4\right) + \cos^2\left(x\right)}
\leqslant 16 \int_{n^{-2/3}}^{n^{-1/3}} \dfrac{dx}{x^2 n^{5/3} \sin^2\left(n ^{-1/3}\right) + 16\cos^2\left(x\right)} \notag \\
& \leqslant 16 \int_{n^{-2/3}}^{n^{-1/3}} \dfrac{dx}{x^2 n^{5/3} \sin^2\left(n ^{-1/3}\right) + 16 - 16 x^2}
\leqslant 16 \int_{n^{-2/3}}^{n^{-1/3}} \dfrac{dx}{x^2 n^{5/3} \sin^2\left(n ^{-1/3}\right) + 16 \left(1- n^{-2/3}\right)}.
\end{align*}
We apply change of variables $ u=c_n x $ with 
$ c_n = \dfrac{n^{5/6}\sin (n^{-1/3})}{4\sqrt{1-n^{-2/3}}} $ to have that
\begin{align*} 
& 16 \int_{n^{-2/3}}^{n^{-1/3}} \dfrac{dx}{x^2 n^{5/3} \sin^2\left(n ^{-1/3}\right) + 16 \left(1- n^{-2/3}\right)}
= \frac{4n^{-5/6}}{\sqrt{1-n^{-2/3}}  \sin\left(n^{-1/3}\right)}
\int_{c_n n^{-2/3}}^{c_n n^{-1/3}} \dfrac{dx}{x^2 +1}
\notag \\
& \leqslant \frac{4n^{-5/6}}{\sqrt{1-n^{-2/3}}  \sin\left(n^{-1/3}\right)}
\int_{c_n n^{-2/3}}^{\infty} \dfrac{dx}{x^2 +1}
= \frac{n^{-5/6}\left( 2\pi - 4\arctan\left(c_n n^{-2/3}\right) \right)}{\sqrt{1-n^{-2/3}}  \sin\left(n^{-1/3}\right)}	.
\end{align*}
Thus, combining all the relations appeared after (\ref{f47}) gives us
\begin{align*} 
& 0 \leqslant \int_{n^{-1/6}}^{n^{1/6}} \dfrac{\cos\left(x/\sqrt{n}\right)\, dx}{n \sin^2\left(x/4\sqrt{n}\right) + \cos^2\left(x/\sqrt{n}\right)} 
\leqslant  \frac{n^{-1/3} \left( 2\pi - 4\arctan\left(c_n n^{-2/3}\right) \right)}{\sqrt{1-n^{-2/3}}  \sin\left(n^{-1/3}\right)}	.
\end{align*}
Note that 
$ 
\dfrac{n^{-1/3}}{\sqrt{1-n^{-2/3}} \sin \left(n^{-1/3}\right)} 
= 1+ \O  \! \left(n^{-2/3}\right)
$
as $ n \to \infty. $
We also know that $ \arctan(x) = \O  \! \left(x\right) $ as $ x\to 0, $ therefore 
$ 
\arctan\left(c_n n^{-2/3}\right) 
= \O  \! \left(c_n n^{-2/3}\right) = \O  \! \left(n^{-1/6}\right)
$
as $ n \to \infty. $
Thus, there exist $ C>0 $ and $ N \in \N $ such that for all $ n \geqslant N $, we have
\begin{align*} 
& 0 \leqslant \int_{n^{-1/6}}^{n^{1/6}} \dfrac{\cos\left(x/\sqrt{n}\right)\, dx}{n \sin^2\left(x/4\sqrt{n}\right) + \cos^2\left(x/\sqrt{n}\right)} 
\leqslant 2\pi + Cn^{-1/6}.
\end{align*}
Now, the last inequality along with (\ref{f47}) guarantees that
\begin{align} \label{f53}
& \int_{n^{-1/6}}^{n^{1/6}} \dfrac{\cos\left(x/\sqrt{n}\right)\, dx}{n \sin^2\left(x/4\sqrt{n}\right) + \cos^2\left(x/\sqrt{n}\right)} 
= 2\pi + \O  \! \left(n^{-1/6}\right).
\end{align}
Hence, (\ref{f39}) follows from (\ref{f46}) and (\ref{f53}).

Finally, we need to show
\begin{align} \label{f54}
\E \left[ N_{n} \left(\Omega_{k,4}\right) \right] 
& = \O\!\left(n^{-1/6}\right) \asntoinfty .
\end{align}
It is quite easy to check that
\begin{align*} 
A(x) 
= \O\!\left(n^{-1/3}\right) + \cos^2(nx) + \O\!\left(n^{-11/6}\right) 
= \cos^2(nx) + \O\!\left(n^{-1/3}\right) \asntoinfty \ \text{and} \ x \in \Omega_{k,4},
\end{align*}
\begin{align*} 
C(x) 
& = \frac{n^3}{16} + \frac{5n^3 \cos^2(n x/4)}{24} 
+ \O  \! \left(n^{13/6}\right) 
= \frac{n^3}{16} + \O\!\left(n^{5/3}\right)
+  \O  \! \left(n^{13/6}\right) \notag \\
& = \frac{n^3}{16} + \O\!\left(n^{13/6}\right)
\asntoinfty \ \text{and} \ x \in \Omega_{k,4}
\end{align*}
and
\begin{align*} 
B(x) 
= - \frac{n^2 \sin(nx/2)}{8} + \O\!\left(n^{7/6}\right)
= \O\!\left(n^{4/3}\right) + \O\!\left(n^{7/6}\right) = \O\!\left(n^{4/3}\right) \asntoinfty \ \text{and} \ x \in \Omega_{k,4}.
\end{align*}
Therefore,
\begin{align*} 
\sqrt{A(x)C(x) - B^2(x)}
& = \sqrt{\frac{n^3\cos^2(nx)}{16} + \O  \! \left(n^{8/3}\right)} 
= \frac{n^{3/2}\cos(nx) \sqrt{1+ \O  \! \left(n^{-1/3}\right)}}{4} \notag \\
& = \frac{n^{3/2}\cos(nx) \left(1+ \O  \! \left(n^{-1/3}\right)\right)}{4} \asntoinfty \ \text{and} \ x \in \Omega_{k,4}.
\end{align*}
Therefore, Proposition \ref{Prop4.1.1} implies that
\begin{align*} 
\E [ N_{n} \left(\Omega_{k,4}\right) ] 
&= \frac{1}{\pi} \displaystyle \int_{\Omega_{k,4}} 
\dfrac{n^{3/2} \left(1+ \O  \! \left(n^{-1/3}\right)\right) dx}
{4 \cos(nx) \left(1+ \O  \! \left(n^{-1/3}\right)\right)}  
=  \dfrac{n^{3/2}\left(1+ \O  \! \left(n^{-1/3}\right)\right) }{4\pi} \displaystyle \int_{\Omega_{k,4}} \dfrac{dx}{\cos(nx)} \notag \\
&  = \O  \! \left(n^{3/2}\right) \displaystyle \int_{\Omega_{k,4}} \dfrac{dx}{\cos(nx)}.
\end{align*}
We implement change of variables $ u=n \left( x-x_k \right) $ and observe that
\begin{align*} 
& \int_{\Omega_{k,4}} \dfrac{dx}{\cos(nx)} 
= \frac{1}{n} \int_{0}^{n^{-2/3}} \dfrac{dx}{\cos(x)} 
= \frac{\ln \left( \sec\left(n^{-2/3}\right) + \tan\left(n^{-2/3}\right) \right) }{n} 
= \O  \! \left(n^{-5/3}\right)
\end{align*}
because $ \ln \left( \sec(x) + \tan(x) \right) = \O  \! \left(x\right) $ as $ x\to 0 . $
Hence, (\ref{f54}) holds. Thus, putting together (\ref{f25}), (\ref{f31}), (\ref{f39}) and (\ref{f54}) implies (\ref{f24}) holds, so does (\ref{f22}). Thus, (\ref{f21}) and (\ref{f22}) lead us to the desired result, namely,
\begin{align} \label{f63}
\E \left[ N_{n} (0,2\pi) \right] 
& = 2 \, \E \left[ N_{n} (0,\pi) \right] 
= 2	\, \E \left[ N_{n} (G \cup F \cup U) \right] = \frac{n}{2} +\frac{\sqrt{13} n}{2\sqrt{3}}
+ \O  \! \left(n^{5/6}\right) \asntoinfty.
\end{align}
Lastly, (\ref{f63}) together with setting $ a=1/6 $ in (\ref{d16-1}) concludes the proof.
\end{proof}

\section*{Acknowledgment}
I am grateful to my advisor Igor Pritsker for leading me in the direction of this project and for all his helpful comments, suggestions and hints which significantly improved the paper. This manuscript is one of the main contributions to the author's Ph.D. dissertation under his supervision.

\bibliographystyle{amsplain}

\end{document}